\documentclass[reqno,12pt,draft]{amsart}
\usepackage{amsmath,amsthm,amssymb}
\usepackage{amsfonts}

\date{}

\newtheorem{theorem}{Theorem}[section]
\newtheorem{definition}[theorem]{Definition}
\newtheorem{lemma}[theorem]{Lemma}
\newtheorem{corollary}[theorem]{Corollary}
\newtheorem{proposition}[theorem]{Proposition}
\newtheorem{example}[theorem]{Example}
\newtheorem{remark}[theorem]{Remark}
\numberwithin{equation}{section}

\begin{document}

\centerline{{\bf Sobolev functions on infinite-dimensional domains}}

\vskip .2in

\centerline{{\bf Vladimir I.~Bogachev$^{a,}$\footnote{Corresponding author, vibogach@mail.ru.},
Andrey Yu.~Pilipenko$^{b}$,}}

\vskip .1cm

\centerline{{\bf and Alexander V.~Shaposhnikov$^{c}$}}

\vskip .2in

$^{a}$ Department of Mechanics and Mathematics, Moscow State University, 119991 Moscow, Russia
and St.-Tikhon's Orthodox University, Moscow, Russia

$^{b}$ Institute of Mathematics NAS, 3, Tereschenkivska st., 01601 Kiev,  Ukraine

$^{c}$ Department of Mechanics and Mathematics, Moscow State University, 119991 Moscow, Russia

\vskip .2in
\centerline{{\bf Abstract}}

\vskip .1in

We study extensions of Sobolev and BV functions on infinite-dimen\-si\-o\-nal domains.
Along with some
positive results we  present a negative solution of the long-standing problem of existence
of  Sobolev extensions of functions in Gaussian Sobolev spaces from a convex domain to the whole space.



\vskip .2in

\section*{Introduction and notation}

In this paper we study several different  Sobolev and BV classes of functions
on infinite-dimensional domains with measures.
Similarly to the finite-dimensional case one can define such classes
by means of suitable completions or via the integration by parts formula.
However, in infinite dimensions new features appear.
The extension problem for these classes is discussed; along with some
positive results we  present a negative solution of the long-standing problem of existence 
of Sobolev extensions of functions in Gaussian Sobolev spaces from a convex domain to the whole space.

In sections 1 and 2, we discuss Sobolev classes on infinite-dimensional domains with Gaussian 
measures; our main result here is Theorem~2.4 that gives Gaussian Sobolev functions 
on convex domains without So\-bo\-lev extensions. In Section~3 we consider more general differentiable 
measures and discuss two types of BV functions on domains; the corresponding main result is Theorem~3.10. 
Note that some of the results and definitions in this section are new also in the Gaussian case; 
in particular, our definition of BV functions based on semivariations of vector measures is a novelty 
also for Gaussian measures. 

Let us recall that, given a domain $U$ in $\mathbb{R}^d$, there are
two common ways of introducing the Sobolev classes of functions
on~$U$.

Let $1\le p<\infty$.
The Sobolev class $W^{p,1}(U)$ consists of all functions
(more precisely, of equivalence classes)
$f\in L^p(U)$ such that the generalized derivatives $\partial_{x_i}f$
(defined by means of the integration by parts formula) are also in $L^p(U)$.
This space is complete with respect to the Sobolev norm
$$
\|f\|_{p,1,U}:=\|f\|_{L^p(U)}+\|\nabla f\|_{L^p(U)},
$$
where $\nabla f:=(\partial_{x_1}f,\ldots,\partial_{x_d}f)$. The same class coincides with the completion
of the space of smooth functions on $U$ with finite norm
$\|\,\cdot\,\|_{p,1,U}$ with respect to this norm.

Similar constructions lead to weighted Sobolev classes if in place of
Lebesgue measure we use a measure $\mu$ with some density $\varrho$ on~$U$.
However, in the latter case the notation $W^{p,1}(U,\mu)$ is reserved for the completion
of the set of smooth functions with respect to the naturally defined
weighted Sobolev norm $\|\,\cdot\,\|_{p,1,U,\mu}$
(certainly, only smooth functions with finite norm
are involved). Moreover, certain additional assumptions about the weight $\varrho$
are needed in order to obtain closable norms (which enables us to regard
$W^{p,1}(U,\mu)$ embedded into $L^p(U,\mu)$), see \cite[Chapter~2]{B10}.
The definition based on the integration by parts formula requires some
other assumptions about the weight $\varrho$ such as its local
membership in the usual Sobolev class; this definition leads to the classes
$G^{p,1}(U,\mu)$, which may be larger than $W^{p,1}(U,\mu)$ even in the case
where both definitions apply. Yet another (but closely related)
 space $D^{p,1}(U,\mu)$ consists of all functions $f\in L^p(U,\mu)$
 with partial derivatives in $L^p(U,\mu)$, but this time these partial derivatives
 are defined as follows: for each $i\le d$ there is a version of $f$ that is
 locally absolutely continuous in $x_i$ and $\partial_{x_i}f$ is defined
 as the partial derivative of this version.
Finally, one can use the class $H^{p,1}(U,\mu)$ consisting of all
functions $f$ in $W^{1,1}_{loc}$ that have finite weighted
norm  $\|\,\cdot\,\|_{p,1,U,\mu}$. The space $H^{p,1}(U,\mu)$ may be strictly
larger than $W^{p,1}(U,\mu)$, see \cite{Zhikov}.
If $U=\mathbb{R}^d$ and $\varrho$ is positive
and locally Lipschitzian, then
$$
W^{p,1}(\mathbb{R}^d,\mu)=G^{p,1}(\mathbb{R}^d,\mu)=D^{p,1}(\mathbb{R}^d,\mu)
=H^{p,1}(\mathbb{R}^d,\mu)
$$
 whenever $1\le p<\infty$.

All the four constructions have natural infinite-dimensional analogues,
see \cite{B10}, \cite{Malliavin97}, \cite{Shig},
where the case of the whole space is discussed
(Sobolev classes over infinite-dimensional Gaussian measures
were first introduced in \cite{Frolov70}, \cite{Frolov72}, some related
objects like $H$-gradient appeared already in the seminal paper \cite{Gross}; later
these classes became very useful in the Malliavin calculus).
In this paper we consider the case that is an analogue of the finite-dimensional
situation with a positive locally Sobolev weight, so that we have two
series of Sobolev spaces $W^{p,1}(U,\mu)$ and $G^{p,1}(U,\mu)$ arising
via completions and via integrations by parts. In addition, we consider two
series of spaces of functions of bounded variation $SBV(U,\mu)$ and $BV(U,\mu)$
defined similarly to the finite-dimensional case (with weights)
and involving for generalized derivatives  vector measures of bounded
semivariation and bounded variation, respectively, which makes them
 distinct in infinite dimensions.

 Most of our discussion concerns extensions of Sobolev and BV functions from domains
 to the whole space.
It is well-known that every function in the Sobolev class on a
bounded convex set in $\mathbb{R}^n$ extends to a function on the
whole space from the same Sobolev class (see \cite[\S 4.4]{EG}, \cite[\S 6.3]{M}
 or \cite[\S 5.10]{Ziemer}). An
analogous assertion is true also for weighted Sobolev classes with
sufficiently regular weights, for example, Gaussian. It has been shown in the
recent paper \cite{H} that for a Gaussian measure on an
infinite-dimensional space a weaker assertion is true: an
extension exists for all functions in some everywhere dense set in
the Gaussian Sobolev space (more precisely, in our notation, in
the class $D^{2,1}(V,\gamma)$  defined through Sobolev
derivatives) on a convex open set (or $H$-convex and $H$-open).
We prove the existence of Sobolev functions
on convex sets in an infinite-dimensional space equipped  with a
Gaussian measure without Sobolev extensions to the whole space. In
the case of a Hilbert space such a set can be chosen convex and
open. Next we introduce classes of functions of bounded variation
on infinite-dimensional convex domains and prove that all functions in such
classes admit BV extensions to the whole space in the case of domains
of bounded perimeter.

Any bounded (possibly signed) measure $\mu$ on
the Borel $\sigma$-algebra $\mathcal{B}(X)$ of a topological space $X$
can be uniquely written as $\mu=\mu^{+}-\mu^{-}$ with mutually singular nonnegative
measures $\mu^{+}$ and $\mu^{-}$. Let
$$
|\mu|:=\mu^{+}+\mu^{-}, \quad \|\mu\|:=|\mu|(X).
$$
We consider only Radon measures (possibly signed). A~bounded measure $\mu$ on
$\mathcal{B}(X)$ is called Radon if for every $B\in\mathcal{B}(X)$
and every $\varepsilon>0$ there is a compact set $K_\varepsilon\subset B$ such that
$|\mu|(B\backslash K_\varepsilon)< \varepsilon$.

For any measure $\mu$ and any function $f$ integrable with respect to $|\mu|$,
let $f\mu$ denote the measure given by the Radon--Nikodym density $f$ with respect to~$\mu$.

We recall  (see \cite{B}) that
a Radon probability measure $\gamma$ on a locally convex space $X$
with topological dual $X^*$ is called centered Gaussian if every functional
$l\in X^*$ is a centered Gaussian random variable on $(X,\gamma)$, i.e.,
$$
\int_X \exp (il)\, d\gamma =\exp\biggl(-\frac{1}{2}\int_X l^2\, d\gamma\biggr).
$$
The Cameron--Martin space $H=H(\gamma)$ of this measure
is the set of all vectors $h\in X$ with $|h|_H<\infty$,
where
$$
|h|_H:=\sup\Bigl\{ l(h)\colon\, l\in X^*, \ \|l\|_{L^2(\gamma)}\le 1\Bigr\}.
$$
This is also the set of all vectors the shifts along which give equivalent measures.

It is known that $h\mapsto |h|_H$ is a Hilbert norm on $H$ and that $H$
is separable.
Let $X_\gamma^*$ denote the closure of $X^*$ in $L^2(\gamma)$.
For every element $h\in H$ there is a unique element $\widehat{h}\in X_\gamma^*$
such that
$$
l(h)=\int_X l \widehat{h}\, d\gamma \quad \forall\, l\in X^*.
$$
Conversely, for every $\xi\in X_\gamma^*$ there is a unique element
$h\in H$ such that $\xi=\widehat{h}$. The element $\widehat{h}$ is called
the measurable linear functional generated by $h$.

We shall deal with the so-called standard
Gaussian measure $\gamma$ on the countable
product of the real lines $X=\mathbb{R}^\infty$ equal to the
countable product of the standard Gaussian measures on the real
line.
Its Cameron--Martin space is the usual $l^2$ and the measurable linear
functionals are described by the formula $\sum_{n=1}^\infty h_nx_n$,
where $(h_n)\in l^2$ and the series converges in $L^2(\gamma)$ and almost everywhere.
By the  Tsirelson isomorphism theorem, every centered Radon Gaussian
measure with an infinite-dimensional Cameron--Martin space is isomorphic
to this standard Gaussian measure by means of a measurable linear isomorphism.
More precisely, this isomorphism can be found in the form
$(x_n)\mapsto \sum_{n=1}^\infty x_n e_n$, where $\{e_n\}$ is an orthonormal
basis in $H(\gamma)$ and the series converges almost everywhere in $X$.

The objects and questions considered below are invariant
with respect to measurable linear isomorphisms of locally convex
spaces with centered Radon Gaussian measures, hence  it would be possible
to deal with the countable power of the standard Gaussian measure on~$\mathbb{R}$
or with the classical Wiener measure (or any infinite-dimensional
Gaussian measure on a Hilbert space).

Let  $\mathcal{F}\mathcal{C}_b^\infty$ denote the
class of all functions on $\mathbb{R}^\infty$ of the form
$$
f(x)=f_0(x_1,\ldots,x_n), \quad f_0\in C_b^\infty (\mathbb{R}^n).
$$
The gradient of $f$ along the subspace $H=l^2$ is defined by the equality $\nabla
f=(\partial_{x_k}f)_{k=1}^\infty$ and is a mapping with values
in~$H$.

Similarly, in the case of a general locally convex space $X$
the symbol $\mathcal{F}\mathcal{C}_b^\infty$ denotes the space of all
functions of the form
$$
f(x)=f_0(l_1(x),\ldots,l_n(x)), \quad f_0\in C_b^\infty (\mathbb{R}^n),
l_i\in X^*.
$$
If $H\subset X$ is a continuously embedded Hilbert space
with an orthonormal basis $\{e_i\}$, then the gradient $D_Hf$
of $f\in \mathcal{F}\mathcal{C}_b^\infty$ along $H$
is defined by the equality
$$
(D_Hf(x),h)_H=\partial_h f(x).
$$
It is readily seen that $D_Hf(x)=\sum_{i=1}^\infty \partial_{e_i}f(x) e_i$.
The series converges in $H$ for every $x$, since
$$
\partial_{e_i}f(x)=\sum_{j=1}^n \partial_{x_j}f_0(l_1(x),\ldots,l_n(x)) l_j(e_i),
$$
so that we have $\sum_{i=1}^\infty |\partial_{e_i}f(x)|^2<\infty$,
because $\sum_{i=1}^\infty |l_j(e_i)|^2<\infty$ for every $j=1,\ldots,n$
by the continuity of $l_j$ on $H$ (which follows by the continuity of the embedding $H\to X$).

\section{Gaussian Sobolev functions on infinite-dimensional domains}

In this section and the next one $\gamma$ is a centered Radon Gaussian measure
on a locally convex space~$X$ and $H$ is its Cameron--Martin space;
without loss of generality one may assume that
$\gamma$ is the countable power of the standard Gaussian measure on the real line and
is defined on~$\mathbb{R}^\infty$, the space of all real sequences. This measure is called
the standard Gaussian product-measure.
Let $\{e_n\}$ be an orthonormal basis in $H$; if $H=l^2$, we use the standard basis of~$l^2$.
Set
$$
x_n:=\widehat{e_n}(x).
$$
In the case of the standard Gaussian product-measure and the standard basis of~$l^2$
the functions $x_n:=\widehat{e_n}(x)$ are just the usual coordinate functions on~$\mathbb{R}^\infty$.

Denote by $\|\,\cdot\,\|_p$ the standard norm in $L^p(\gamma)$;
the same notation is used for the norm in the space
$L^p(\gamma,H)$ of measurable mappings $v$ with values in~$H$ such
that $|v|_H\in L^p(\gamma)$. For $p\in [1,+\infty)$ the Sobolev
class $W^{p,1}(\gamma)$ is defined as the completion of
$\mathcal{F}\mathcal{C}_b^\infty$ with respect to the Sobolev norm
$$
\|f\|_{p,1}=\|f\|_{p}+\|\nabla f\|_p.
$$
Each  function $f$ in $W^{p,1}(\gamma)$  has the Sobolev gradient
 $\nabla f=(\partial_{e_n}f)$ in $L^p(\gamma,H)$ whose components $\partial_{e_n}f$ satisfy the identity
$$
\int \varphi(x) \partial_{e_n}f(x) \, \gamma (dx) =-\int
[f(x)\partial_{e_n}\varphi(x)-x_n f(x)\varphi(x)]  \, \gamma (dx)
$$
for all $\varphi\in \mathcal{F}\mathcal{C}_b^\infty$. This definition of the
class $W^{p,1}(\gamma)$ is equivalent to the following one (see
\cite{B}, \cite{B10}): $f\in W^{p,1}(\gamma)$ precisely when $f\in
L^p(\gamma)$ and, for every fixed $n$, the function $f$
 has a version  (i.e., an almost everywhere equal function) $\widetilde{f}$ such that the
 functions $t\mapsto \widetilde{f}(x+te_n)$, where $x\in X$, are absolutely continuous
 on each line and the mapping $\nabla f=\sum_{n=1}^\infty \partial_{e_n} \widetilde{f}e_n$
  belongs to $L^p(\gamma,H)$; here
$\partial_{e_n} \widetilde{f}(x)$ is defined as the derivative at
zero of the function $t\mapsto \widetilde{f}(x+te_n)$ (one can
show that it exists $\gamma$-a.e.). Actually the definition of the
Sobolev class does not depend on our choice of an orthonormal
basis in the  Cameron--Martin  space.

In a particular way one introduces the space $BV(\gamma)$ of
functions of bounded variation, containing $W^{1,1}(\gamma)$, see
 \cite{Fuk}, \cite{FH}, \cite{Hino04}, \cite{Hino09}, \cite{H},
 \cite{A}, \cite{A2}.
 It consists of functions $f\in L^1(\gamma)$ such that
$x_n f\in L^1(\gamma)$ for all $n$ and there is an  $H$-valued
measure
 $\Lambda f$ of bounded variation for which
the scalar measures $(\Lambda f,e_n)_H$ satisfy the identity
$$
\int \varphi(x)\, (\Lambda f,e_n)_H(dx) =-\int
[f(x)\partial_{e_n}\varphi(x)-x_n f(x)\varphi(x)]  \, \gamma (dx)
$$
for all $\varphi\in \mathcal{F}\mathcal{C}_b^\infty$. If $f\in W^{1,1}(\gamma)$, then
$\Lambda f=\nabla f\cdot\gamma$. The space $BV(\gamma)$ is Banach
with the norm
$$
\|f\|_{1,1}=\|f\|_1+\|\Lambda f\|,
$$
 where
$\|\Lambda f\|$ is the variation of the vector measure $\Lambda f$
defined by
$$
\|\Lambda f\|=\sup\sum_{i=1}^\infty |\Lambda
f(B_i)|_H,
$$
 where sup is taken over all partitions of the space
into disjoint Borel parts~$B_i$. Let us note that there is $C>0$
such that
$$
\| \widehat{h}f\|_{L^1(\gamma)}\le C\|f\|_{1,1},\quad f\in
BV(\gamma),
$$
where $\widehat{h}(x)=\sum_{n=1}^\infty h_n x_n$,
$\sum_{n=1}^\infty h_n^2\le 1$.

Let us proceed to domains.

Suppose now that we are given a Borel or $\gamma$-measurable set
$V\subset X$ of positive $\gamma$-measure such that  its
intersection with every straight line of the
 form $x+\mathbb{R}^1 e_n$  is a convex set~$V_{x,e_n}$.

If the sets $(V-x)\cap H$ are open in $H$ for all $x\in V$, then
$V$ is called $H$-open (this property is equivalent to the fact
that $V-x$ contains a ball from~$H$ for every $x\in V$, and is
weaker than openness of $V$ in~$X$), and if all such sets are
convex, then $V$ is called $H$-convex. The latter property is
weaker than the usual convexity. Obviously, for any $H$-convex and $H$-open set $V$
all nonempty sets $V_{x,e_n}$ are open intervals (possibly unbounded).

\begin{example}
\rm
Any open convex set is $H$-open and $H$-convex.
However, the convex ellipsoid
$$
U=\Bigl\{x\in\mathbb{R}^\infty\colon\, \sum_{n=1}^\infty n^{-2} x_n<1\Bigr\}
$$
is not open in $\mathbb{R}^\infty$, but is $H$-open, where $H=l^2$. This ellipsoid has
positive measure with respect to the standard Gaussian product-measure~$\gamma$.
The set
$$
Z=\Bigl\{x\in\mathbb{R}^\infty\colon\, \lim\limits_{N\to\infty}
N^{-1} \sum_{n=1}^N  x_n^2=1\Bigr\}
$$
is Borel and has full measure with respect to $\gamma$ (by the law of large numbers).
It is not convex: if $x\in Z$, then $-x\in Z$, but $0\not\in Z$. It is also not difficult
to see that $Z$ has no interior. However,
$Z$ is $H$-open and $H$-convex, since for every $x\in Z$ we have $(Z-x)\cap H=H$, that is,
for every $h\in H$ we have $x+h\in Z$. Indeed,
$$
N^{-1} \sum_{n=1}^N  (x_n+h_n)^2-N^{-1} \sum_{n=1}^N  x_n^2
=N^{-1} \sum_{n=1}^N  (h_n^2+2x_nh_n),
$$
which tends to zero as $N\to\infty$, because $N^{-1} \sum_{n=1}^N  h_n^2\to 0$ for each $h\in H$
and $|N^{-1}\sum_{n=1}^N  x_nh_n|\to 0$ by the Cauchy inequality.
\end{example}

There are several natural ways of introducing Sobolev classes
on~$V$. The first one is considering the class $W^{p,1}(V,\gamma)$
equal to the completion of $\mathcal{F}\mathcal{C}_b^\infty$ with respect to the Sobolev
norm $\|\,\cdot\,\|_{p,1,V}$ with the order of integrability~$p$,
evaluated with respect to the restriction of $\gamma$ to~$V$. This
class is contained in the class $D^{p,1}(V,\gamma)$ consisting of
all functions $f$ on~$V$ belonging to $L^p(V,\gamma)$ and having
versions of the type indicated above, but with the difference that
now the  absolute continuity is required only on the closed
intervals belonging to the sections $V_{x,e_n}$. The class
$D^{p,1}(V,\gamma)$ is naturally equipped with the Sobolev norm
$\|\,\cdot\,\|_{p,1,V}$ defined by the restriction of $\gamma$
to~$V$:
$$
\|f\|_{p,1,V}=\biggl(\int_V |f|^p\, d\gamma\biggl)^{1/p}+
\biggl(\int_V |\nabla f|^p\, d\gamma\biggl)^{1/p}.
$$
In the paper \cite{H} the Sobolev class $D^{2,1}(V,\gamma)$ was used  (denoted
there by $W^{1,2}(V)$). In the finite-dimensional case for convex
$V$ both classes coincide, the infinite-dimensional situation is
less studied, but for $H$-convex $H$-open sets one has
$W^{2,1}(V,\gamma)=D^{2,1}(V,\gamma)$, which follows from \cite{H}, where
it is shown that $D^{2,1}(V,\gamma)$ contains a dense set of functions possessing
extensions of class $W^{2,1}(\gamma)$ and for this reason belonging
to $W^{2,1}(V,\gamma)$. It
is readily verified  that the spaces $W^{p,1}(V,\gamma)$ and
$D^{p,1}(V,\gamma)$ with the Sobolev norm are Banach.

Similar Sobolev classes can be defined for more general differentiable
measures (see \cite{B10}, \cite{B97}, \cite{BM} and Section~3 below).

\section{Gaussian Sobolev functions without extensions}

In this section we prove the existence of functions in
$W^{p,1}(V,\gamma)$ that have no Sobolev extension to the whole
space. We provide a complete justification of the example announced in \cite{BSh},
where details have been omitted, also correcting some estimates; in addition, a rather explicit
function with the stated property is given.
Note that one can introduce more narrow Sobolev classes on
$V$ that admit extensions. For example, in the space
$W^{p,1}(V,\gamma)$ one can take the closure $W^{p,1}_0(V,\gamma)$
of the set of functions from $W^{p,1}(\gamma)$ with compact
support in~$V$; the functions from $W^{p,1}_0(V,\gamma)$ extended
by zero outside~$V$ belong to $W^{p,1}(\gamma)$. For certain very
simple sets~$V$, say, half-spaces, it is easy to define
explicitly an extension operator.

\begin{example}
\rm
Let $V=\{x\colon\, \widehat{h}(x)>0\}$, where $|h|_H=1$.
 Any  function $f\in W^{p,1}(V,\gamma)$ has an extension
of class $W^{p,1}(\gamma)$ defined by
$$
\widetilde{f}(x)=f(x-2\widehat{h}(x)h)\quad \hbox{if } \widehat{h}(x)<0.
$$
For example, if $X=\mathbb{R}^\infty$ and $h=e_1$, then $V=\{x_1>0\}$
and $\widetilde{f}(x_1,x_2,\ldots)=f(-x_1,x_2,\ldots)$ whenever $x_1<0$.
Taking a sequence of functions $f_j\in \mathcal{F}\mathcal{C}_b^\infty$ whose restrictions
to $V$ converge to $f$ in the norm of $W^{p,1}(\gamma,V)$, we see that the functions
$\widetilde{f_j}$ redefined on the set $\{\widehat{h}\le 0\}$ by
$\widetilde{f_j}(x)=f_j(x-2\widehat{h}(x)h)$ belong to $W^{p,1}(\gamma)$
and converge in the Sobolev norm.
In particular, $\|\widetilde{f_j}\|_{L^p(\gamma)}=2^{1/p}\|f_j\|_{L^p(V,\gamma)}$,
$\|D_H\widetilde{f_j}\|_{L^p(\gamma)}=2^{1/p}\|D_H f_j\|_{L^p(V,\gamma)}$.
\end{example}

It is not clear
whether there are essentially infinite-dimensional domains~$V$ for
which all Sobolev functions  have extensions.

However, functions with bounded derivatives extend from any $H$-convex domains.

\begin{proposition}
Let $V$ be $H$-convex and let $f\in W^{2,1}(\gamma, V)$ be such that $|D_H f(x)|\le C$ a.e. in~$V$.
Then there is a function $g\in W^{p,1}(\gamma)$ for all $p<\infty$
such that $g|_V=f|_V$ a.e. on~$V$ and
$|g(x+h)-g(x)|\le C |h|_H$ for all $x\in X$ and $h\in H$.
\end{proposition}
\begin{proof}
Suppose first that $|f(x)|\le N$ for some~$N$.
Let $\{e_n\}$ be an orthonormal basis in $H$. Let $H_n$ be the linear span of $e_1,\ldots,e_n$.
By using conditional measures it is readily
seen that for every $n$ there is a version $f_n$ such that
$|f_n(x+h)-f_n(x)|\le C |h|_H$ for all $x\in V$ and $h\in H_n$ with $x+h\in V$.
One can find a measurable function $g_n$ extending $f_n$
to all of $X$ such that $|g_n(x)|\le N$ and
$$|g_n(x+h)-g_n(x)|\le C |h|_H \quad \hbox{for all $x\in X$ and $h\in H_n$,}
$$
 see
\cite[Theorem~5.11.10]{B} or~\cite{B99}. Passing to a subsequence we may assume that
the sequence $(g_1+\cdots+g_n)/n$ converges in $L^2(\gamma)$ to some function~$g$. Obviously,
$|g(x)|\le N$ and $g=f$ a.e. on~$V$.
It is straightforward to verify that $g\in W^{2,1}(\gamma)$ and $|D_Hg|_H\le C$.
Then we have $g\in W^{p,1}(\gamma)$ for all $p<\infty$.
In addition, $g$ has a version (denoted by the same symbol) such that
$|g(x+h)-g(x)|\le C |h|_H$ for all $x\in X$ and $h\in H$, see \cite[Corollary~5.11.8]{B}.

In order to remove the restriction that $|f|\le N$ we apply the considered case to the
functions $\psi_N(f)$, where $\psi_N\in C^\infty(\mathbb{R})$, $\psi_N(t)=t$ if $|t|< N$,
$\psi_N(t)=N {\rm sign}\, t$ if $|t|\ge N$. It is easily verified that $\psi_N(f)\in W^{2,1}(V,\gamma)$
and $|D_H \psi_N(f)|_H\le C$. We obtain functions $g_N\in W^{2,1}(\gamma)$ such that $g_N|V=f|_V$ a.e. on~$V$,
$|g_N|\le N$ and $|g_N(x+h)-g_N(x)|\le C |h|_H$ for all $x\in X$ and $h\in H$.
We observe that the integrals of $g_N$, denoted by $I(g_N)$, are uniformly bounded.
Indeed, otherwise, passing to a subsequence, we may assume that $|I(g_N)|\to +\infty$.
The integrals of $|g_N-I(g_N)|$ are uniformly bounded due to the estimate $|D_H g_N|_H\le C$
and the Poincar\'e inequality,
see, e.g., \cite[Theorem~5.5.1]{B}. Since $g_N(x)\to f(x)$ for every $x\in V$,
we have $|g_N(x)-I(g_N)|\to \infty$ on~$V$, which is impossible by Fatou's theorem.
Applying the Poincar\'e inequality again we see that $\{g_N\}$ is bounded in $L^2(\gamma)$,
hence the same reasoning as in the previous step yields the desired extension.
\end{proof}

\begin{lemma}\label{lem2.1}
If the set $V$ is such that for some $p>1$ every function $f\in
W^{p,1}(V,\gamma)$ has an extension $g\in W^{p,1}(\gamma)$, then
there exists an extension $g_f\in W^{p,1}(\gamma)$ such that
$\|g_f\|_{p,1}\le C \|f\|_{p,1,V}$ with some common constant~$C$.
If the condition is fulfilled for $p=1$, then the conclusion is
true with $g_f\in BV(\gamma)$.
\end{lemma}
\begin{proof}
For every $f\in W^{p,1}(V,\gamma)$ we denote by $C(f)$ the infimum
of numbers $C\ge 0$ for which $f$ has an extension $g$ with
$\|g\|_{p,1}\le C\|f\|_{p,1,V}$. We observe that there exists an
extension with $C=C(f)$. Indeed, take an extension $g_n$ for which
$$
\|g_n\|_{p,1}\le (C(f)+n^{-1})\|f\|_{p,1,V} .
$$
If $p>1$, then the space $W^{p,1}(\gamma)$ is reflexive, hence one
can  pass to a subsequence in~$\{g_n\}$ with the arithmetic means
convergent in $W^{p,1}(\gamma)$, which gives the desired
extension.

If $p=1$, then by the Komlos theorem  (see \cite[p.~290]{B07}) one can find a
subsequence with the arithmetic means convergent almost
everywhere. This gives an extension in the class $BV(\gamma)$,
since if functions $f_j$ are uniformly bounded in $BV(\gamma)$ and
converge a.e. to a function~$f$, then $f\in BV(\gamma)$ (see \cite{FH}).

Thus, for every $f\in W^{p,1}(V,\gamma)$ we obtain the set $E(f)$
of all extensions with the minimal possible norm. It is clear
 that this set is convex and closed in $W^{p,1}(\gamma)$.
If $p>1$, then  it consists of a single point because of the
strict convexity of the norm in~$L^p$.

For any natural number $N$ let
$$
S_N=\bigl\{f\in W^{p,1}(V,\gamma)\colon \ C(f)\le N\bigr\}.
$$
 It is easy to see that $S_N$
is a closed set. By Baire's theorem, there
is   $N$ such that   $S_N$  contains  a closed ball $B(f_0,r)$
in $W^{p,1}(V,\gamma)$ of radius $r>0$ centered at~$f_0$.
Then the
ball $B(0,r)$ is contained in $S_{N_1}$, where $N_1=Nr^{-1}(2\|f_0\|_{p,1,V}+r)$.
Indeed, if $\|f\|_{p,1,V}=r$, then $f$ has an extension $g_{f+f_0}-g_{f_0}$
with the norm estimated by
$$
N(r+\|f_0\|_{p,1,V})+N \|f_0\|_{p,1,V}=N_1r.
$$
Hence the required estimate holds for any $f\in B(0,r)$,
whence the assertion of the lemma  follows.
\end{proof}

The main negative results about extensions are the following two theorems.
For notational simplicity, in the first one we consider the Gaussian product-measure
on~$\mathbb{R}^\infty$.

\begin{theorem}\label{t2.4}
The space $\mathbb{R}^\infty$ contains a convex Borel $H$-open
set~$K$ of positive $\gamma$-measure with the following
property{\rm:} for every $p\in [1,+\infty)$ there is a function in
the class $W^{p,1}(K,\gamma)$ having no extensions to a function
of class~$W^{p,1}(\gamma)$. One can also find a convex compact
set $K$ with the same property.
\end{theorem}
\begin{proof}
Let us fix a natural number $m$ and consider  the open rhomb
$K_m$ in the plane with vertices at the points $a=(m^2,0)$,
$b=(0,m)$, $c=-a$ and $d=-b$. The function $f_m(x)=\max
(1-m^2|x-a|,0)$, where $|x|$ is the usual norm on the plane, is
Lipschitzian. Let us estimate its norm in the space
$W^{p,1}(K_m,\gamma)$, where $\gamma$ is the standard Gaussian
measure on the plane with density $\varrho$; in what follows in the notation
like $\|f\|_{p,1,K_m}$ we always mean Sobolev norms of functions on $K_m$
with respect to the standard Gaussian measure.
To this end we
observe that whenever $|x-y|\le \min(1,1/|x|)$ the inequality
\begin{equation}\label{e1}
c_1\le\varrho(x)/\varrho(y)\le c_2, \quad c_1=e^{-1}, c_2=e^{3/2}
\end{equation}
holds.
In the subsequent estimates we denote by $c_3, c_4,\ldots$ some universal constants.
We have
$$
|\nabla f_m(x)|=m^2 \quad\hbox{if $|x-a|< m^{-2}$,} \quad |\nabla
f_m(x)|=0 \quad \hbox{if $|x-a|> m^{-2}$.}
$$
The domain in $K_m$, where $f_m$ does not vanish, is a sector
with vertex~$a$ and the angle with tangent $1/m$. Hence  by
(\ref{e1}) there hold the inequalities
\begin{equation}\label{e2}
c_3^{-1} \varrho(a)^{1/p}m^{2-5/p}\le
 \|f_m\|_{p,1,K_m}\le c_3 \varrho(a)^{1/p}m^{2-5/p}.
\end{equation}

Now observe that if a
locally Sobolev function $g$ is an extension of $f_m$ to the plane
(or in the case $p=1$ the extension belongs to $BV(\gamma)$), then
already for the open disc  $U$ with center $a$ and radius
$2m^{-2}$ there holds the estimate
\begin{equation}\label{e3}
\|g\|_{1,1,U}\ge c_4 m \|f_m\|_{1,1,K_m}\ge c_5 \varrho(a)m^{-2}.
\end{equation}
Indeed, by
inequality (\ref{e1}) it suffices to obtain such an estimate for
the usual Sobolev norms with respect to  Lebesgue measure without
Gaussian weight.
Moreover, since $|\nabla f_m(x)|\ge f_m(x)$, $x\in K_m$, it is enough to obtain an estimate
of this sort just for the $L^p$-norms of the gradients (in place of the full Sobolev norm).
It is known (see \cite[\S 5.5]{EG} or \cite[\S 6.1]{M}) that for
every Sobolev function $g$ on $U$ one has the equality
$$
\|\nabla g\|_{L^1(U)}=\int_{-\infty}^{+\infty} P(E_t)\, dt,
$$
where
$$
E_t=\{x\in U\colon\, g(x)>t\}
$$
and $P(E_t)$ is the perimeter
of the set~$E_t$, which for almost all $t$ equals $H_{1}(U\cap
\partial E_t)$, where $H_{1}$ is the one-dimensional Hausdorff
measure (the length) and $\partial E_t$ is the boundary of the
set~$E_t$. An analogous equality is true for the function $f_m$ on
$K_m$. For a function $g\in BV(U)$ the indicated equality has the
form
$$
\|Dg\|(U)=\int_{-\infty}^{+\infty} P(E_t)\, dt,
$$
where $\|Dg\|(U)$ is the value of the total variation of the
vector measure $Dg$ (the generalized gradient of~$g$) on the
set~$U$. One can replace $g$ by the function $\min(1,\max(g,0))$
with values in $[0,1]$,  which extends $f_m$ and whose Sobolev norm
does not exceed that of~$f$. Hence we can consider only the values
$t\in [0,1]$. Now it suffices to verify that for the sets
$S_t=\{x\in K_m\colon\, f_m(x)>t\}$ we have the estimate
$$
H_{1}(U\cap \partial E_t)\ge c m H_{1}(K_m \cap \partial S_t)
$$
with a universal  constant~$c$.  The set $K_m\cap \partial S_t$ is
an arc of the circle of radius $m^{-2}(1-t)$ centered at~$a$. If
the set $\partial E_t$ oversteps the limits of~$U$, then the
length of $U\cap \partial E_t$ is not less than~$m^{-2}$ and the
length of the arc
 $K_m\cap \partial S_t$ is not greater than~$cm^{-3}$. If the
set $\partial E_t$ is entirely contained in~$U$, then again its length
is not less than $cm H_1 (K_m\cap \partial S_t)$, since $E_t$
contains the whole sector $S_t$.

From estimates (\ref{e2}) and (\ref{e3}) we obtain
$$
\frac{\|g\|_{1,1,U}}{\|f_m\|_{p,1,K_m}}\ge c_6 \varrho(a)^{1-1/p}m^{5/p-4}.
$$
 By H\"older's inequality
$$
\|g\|_{1,1,U} \le
 \|g\|_{p,1,U}\gamma(U)^{1/q} \le c^{-1}_7\varrho(a)^{1/q}m^{-4/q} \|g\|_{p,1,U}, \quad q=p/(p-1),
$$
hence
$$
\frac{\|g\|_{p,1,U}}{\|g\|_{1,1,U}}\ge c_7 \varrho(a)^{1/p-1} m^{4-4/p}.
$$
 Therefore, we finally obtain
$$
\frac{\|g\|_{p,1,U}}{\|f_m\|_{p,1,K_m}}\ge c m^{1/p}
$$
with some constant $c>0$.

Now in the infinite-dimensional case
we regard $\mathbb{R}^\infty$ as the countable product of the plane~$\mathbb{R}^2$
(so the elements of $\mathbb{R}^\infty$ are sequences $x=(x_n)$ with $x_n\in \mathbb{R}^2$,
$|x_n|$ is the usual norm in the plane)
equipped with the countable power $\gamma$ of the standard Gaussian measure on the plane.

We take for $K$ the
intersection of the product  $\prod_{m=1}^\infty K_m$ with the
linear subspace $L$ consisting of all elements $x=(x_m)$ such that
$$
\sum_{m=1}^\infty m^{-2}|x_m|^2<\infty
$$
 and having full measure; the equality $\gamma(L)=1$ follows from the fact that the integral
of~$x_m^2$ equals~$1$. This subspace is Hilbert with respect to
the norm
$$
\|x\|_L=\Bigl(\sum_{m=1}^\infty
m^{-2}|x_m|^2\Bigr)^{1/2}.
$$

In order to verify that $K$ is $H$-open it suffices to show that
$K$ is open in the Hilbert space~$L$, which is done as follows.
Let $x=(x_m)\in K$. Suppose that for every $n$ there is an element
$h^n=(h^n_m)\in L$ for which $\|h^n\|_L<1/n$ and $x+h^n\not\in K$.
Since $x\in L$, we have $|x_m|\le m/4$ for all~$m$, starting from
some number $m_1$. Then $x_m+z\in K_m$ if $m\ge m_1$ and $|z|\le
m/4$, since $K_m$ contains the ball of radius~$m$. Hence
$x_m+h^n_m\in K_m$ for all~$m$ and sufficiently large~$n$, whence
$x+h^n\in K$, which is a contradiction. It is clear that
$\gamma(K)>0$, since the Gaussian measures of the sets $K_m$ are
rapidly increasing to~$1$.

Every  function $f_m$ constructed above on the plane generates a
function (denoted by the same symbol) on the space
$\mathbb{R}^\infty$, identified with the countable power of the
plane, which acts by the formula $f_m(x)=f_m(x_m)$. The measure
$\gamma$ on $\mathbb{R}^\infty$ is also considered as the
countable power of the standard Gaussian measure on the plane. By
Lemma~\ref{lem2.1} it suffices to show that $C(f_m)\ge c m^{1/p}$. It
remains to observe that if a function $g\in W^{p,1}(\gamma)$
extends~$f_m$, then for almost every
(with respect to the corresponding product-measure)
fixed $y=(y_j)_{j\not=m}$ the
function $g_m(x_m)=g(y_1,\ldots,y_{m-1},x_m,y_{m+1},\ldots)$
of the remaining variable $x_m$ gives an extension
of the function $f_m(x_m)$ and hence one has the estimate
$$
\int |\nabla g|^p\, d\gamma\ge c_0 m  \int |\nabla f_m|^p\,
d\gamma
$$
with some constant~$c_0$, which gives the required assertion.

It is also possible to give an explicit example of a function
$f$ in $W^{p,1}(K,\gamma)$ without extensions to a function
in $W^{p,1}(\gamma)$. Take a sequence $\{C_m\}$ such that
$$
\sum_mC_m\|f_m\|_{p,1,K_m}<\infty,\quad \sup_mC_mm^{1/p}\|\nabla f_m\|_{p,K_m}=\infty,
$$
which is obviously  possible.
If $g$ is an extension of
$$f(x)=\sum_{m=1}^\infty C_mf_m(x_m),
$$
 then
\begin{multline*}
\|g\|_{p,1} \ge \sup_m\biggl(\int\int|\nabla_{x_m} g|^{p}
\,d\gamma_{m}\,d\gamma_{m}^{\bot}\biggr)^{1/p}
\\
\ge
 \sup_m\biggl(\int C_{m}^pc^p m\|\nabla f_m\|_{p,K_m}\,d\gamma_{m}^{\bot}\biggr)^{1/p} = \infty,
\end{multline*}
where $\gamma_{m}^{\bot}$ is the product-measure complementing $\gamma_m$ to~$\gamma$.
Thus, $f$ possesses the desired property.
\end{proof}

\begin{remark}
{\rm
The set $\prod_{m=1}^\infty K_m$ is not $H$-open.
It suffices to consider its point $x=(x_m)$ with $x_m=(m^2-2^{-m},0)$.
}\end{remark}

\begin{remark}
{\rm Passing to the restriction of the measure $\gamma$ to the
Hilbert space~$L$, we obtain a convex and open in $L$  set $K$ of
positive measure, on which for every $p\in [1,+\infty)$ there is a
function in the class $W^{p,1}(K,\gamma)$ without restrictions to
a function in $W^{p,1}(\gamma)$. It is clear that the same example
can be realized also on a larger weighted Hilbert space of
sequences, in which $K$ will be precompact. Hence it is possible to
combine $H$-openness of $K$ with its relative compactness in a Hilbert space (clearly,
our set $K$ in $\mathbb{R}^\infty$ is relatively compact).

In addition, if we take for $\gamma$ the classical Wiener measure on the space
$C[0,1]$ or $L^2[0,1]$ and embed this space into
$\mathbb{R}^\infty$ by means of the mapping $x\mapsto n (x,e_n)_{L^2}$,
where $\{e_n\}$ is the orthonormal basis in $L^2[0,1]$ formed by the eigenfunctions
of the covariance operator of the Wiener measure,
that is, $e_n(t)=c_n\sin ((\pi n-\pi/2)t)$,
$n\in\mathbb{N}$, with the eigenvalues $\lambda_n=(\pi n-\pi/2)^{-2}$,
then the image of $\gamma$ will coincide with the standard Gaussian product-measure and
the space $L$ described above will
coincide with~$L^2[0,1]$ itself
(more precisely, with the image of $L^2[0,1]$ under the embedding),
hence our convex set $K$ will be open in the corresponding space.

For any centered Radon Gaussian measure $\gamma$ on a
locally convex space~$X$ with the infinite-dimensional
Cameron--Martin  space~$H$, the results proved above yield
existence of an $H$-open convex Borel set $V$ of positive
$\gamma$-measure and, for every $p\in [1,+\infty)$, a function
$f$ in $W^{p,1}(V,\gamma)$ without extensions to functions in the
class $W^{p,1}(\gamma)$. It would be interesting to construct an
example of a function in the intersection of all
$W^{p,1}(V,\gamma)$ without extensions of class
$W^{1,1}(\gamma)$. Apparently, there are bounded functions with such a property.
}\end{remark}

Certainly, it is natural to ask about such examples on a ball in a Hilbert space.
However, we have no such examples.

\section{Classes $BV$ on infinite-dimensional domains with differentiable measures}

In the paper \cite{BR} certain classes of
functions of bounded variation on infinite-dimensional
spaces with non-Gaussian differentiable measures have been considered. Earlier such classes in the case of
Gaussian measures were studied in the papers  \cite{FH}, \cite{H}, \cite{A}.
 As in \cite{BR},
here we consider the two types of functions of bounded variation corresponding
to vector measures of bounded
variation or bounded semivariation arising as derivatives.
However, some other nuances also appear due to boundary effects.

We recall that on a domain $U$ in $\mathbb{R}^d$ with  Lebesgue measure the
 class $BV(U)$ is defined as the set of all functions~$f$ integrable in $U$ such that the generalized
gradient $Df$ is a vector measure on $U$
of bounded variation. Here we identify almost everywhere equal functions;
 the class $BV(U)$ turns out to be a Banach space with norm
 $\|f\|_{BV}=\|f\|_{L^1(U)}+{\rm Var}\, (Df)$.
On the real line with the standard Gaussian measure $\mu=\varrho dx$,
$\varrho(x)=(2\pi)^{-1/2}\exp(-x^2/2)$,  the class $BV(\mu)$
consists of all functions $f\in L^1(\mu)$ such that the function $xf(x)$ also belongs to
$L^1(\mu)$ and the generalized derivative of the function $f\varrho$ is a bounded Borel measure
on the real line. In this case one can introduce a bounded measure
$\Lambda f$ on the real line for which the measure $(f\varrho)'$
is the sum of $\Lambda f$ and the measure with density $-x f(x)\varrho(x)$ with respect to Lebesgue
measure. In \cite{BR}, certain natural analogs of these classes for measures on infinite-dimensional
spaces have been introduced. The principal feature of the  infinite-dimensional case
is connected with the fact  that here one can consider
vector measures of bounded variation as well as more general vector measures of
bounded semivariation. Passing to domains we face an additional problem of
defining generalized derivatives, since here domains are understood not in
the topological sense any more. The latter is explained by that fact that many typical
convex sets of positive measure have no inner points.

A measure $\mu$ on $X$ is called differentiable along a vector~$h$
in the sense of Skorohod if there exists a  measure $d_h\mu$, called the Skorohod
derivative of the measure $\mu$ along the vector $h$, such that
$$
   \lim_{t \to 0} \int_X \frac{f(x - th) - f(x)}{t} \, \mu(dx) =
  \int_X f(x) d_h\mu(dx)
$$
for every bounded continuous function $f$ on $X$. It is worth noting that even if $\mu$
is nonnegative (as is the case below), its derivative $d_h\mu$ is always a signed measure: $d_h\mu(X)=0$,
which follows by taking $f=1$.

If the measure $d_h\mu$ is absolutely continuous with respect to the measure~$\mu$, then
 $\mu$ is called Fomin differentiable along the vector $h$, the  Radon--Nikodym density of the
measure $d_h\mu$ with respect to $\mu$ is denoted by $\beta_h^\mu$ and
called the logarithmic derivative of $\mu$ along $h$.
The Skorohod differentiability of $\mu$ along $h$ is equivalent to the identity
$$
\int_X \partial_h f(x)\, \mu(dx)=-\int_X f(x)\, d_h\mu(dx), \quad f\in \mathcal{F}\mathcal{C}^\infty,
$$
where $\partial_h f(x)=\lim\limits_{t\to 0} (f(x+th)-f(x))/t$.
On the real line the Fomin differentiability is equivalent to the membership of the density in
the Sobolev class $W^{1,1}$ and the Skorohod differentiability is the boundedness of
variation of the density (and the Skorohod derivative is just the derivative
of the density in the sense of generalized functions).
On differentiable measures, see \cite{B10}, \cite{B97}, \cite{BS}.

For example, if $\mu$ is a Gaussian measure,
then it is Fomin differentiable along all vectors in its Cameron--Martin
space $H(\mu)$.

It will be useful below that the collection $D_C(\mu)$ of all vectors of Skorohod
differentiability of a nonzero measure $\mu$ is a Banach space with respect to the
norm $h\mapsto \|d_h\mu\|$.
The set $D(\mu)$ of all vectors of Fomin differentiability
is its closed linear subspace with the same norm;
on these results, see \cite[Chapter~5]{B10}, \cite{B85}, \cite{BS}.
It will be useful below and follows by the closed graph theorem that, for any
Hilbert space $H$ continuously embedded into $X$ and contained in~$H$,
 the natural embedding $H\to D_C(\mu)$
is continuous with respect to the indicated norm on $D_C(\mu)$ (and the given norm on~$H$).

We recall some concepts related to conditional measures (see \cite{B07}).
Suppose we are given a  measure $\nu$ (possibly signed)
on $X$ and a vector $h\in X$. Let us choose a
closed hyperplane $Y\subset X$ for which $X=Y\oplus \mathbb{R}h$,
 and denote  by $|\nu|_Y$ the projection of $|\nu|$ on~$Y$ under the natural projecting
$\pi\colon\, X\to Y$. Then one can find measures $\nu^{y,h}$ on the straight
 lines $y+\mathbb{R}h$ (but regarded also as measures on the whole space~$X$), $y\in Y$,
 called conditional measures,
 for which the equality $\nu=\nu^{y,h}|\nu|_Y(dy)$ holds in the sense of the identity
$$
\nu(B)=\int_Y \nu^{y,h}(B)\, |\nu|_Y(dy)
$$
for every Borel set $B$. In terms of the measure $|\nu|$ itself  this can be written as
$$
\nu(B)=\int_X \nu^{\pi x,h}(B)\, |\nu|(dx).
$$
We can also assume that we are given conditional measures $\nu^{x,h}$
on the straight lines $x+\mathbb{R}h$ such that $\nu^{x,h}=\nu^{\pi x,h}$.
If the measure $\nu^{x,h}$ has a density with respect to the natural Lebesgue measure  on $x+\mathbb{R} h$
induced by the mapping $t\mapsto x+th$, then this density is called conditional and is denoted
by $\varrho_\nu^{x,h}$. It is important to note that in place of the measure $|\nu|_Y$ on $Y$
one can use any nonnegative measure $\sigma$ with respect to which the measure $|\nu|_Y$
is absolutely continuous. If $|\nu|_Y=g\sigma$, then we obtain the representation
$\nu=\nu^{y,h,\sigma} \, \sigma(dy)$, where $\nu^{y,h,\sigma}:=g(y)\nu^{y,h}$.
Such a representation is  called a disintegration of the measure.
This may be convenient for a simultaneous disintegration of several measures.
As above, if $\sigma$ is given on~$X$  and $\nu\ll\sigma$, then we can write $\nu=\nu^{x,h,\sigma}\, \sigma(dx)$.

For a Gaussian measure, the conditional measures are Gaussian as well
(see \cite[Section~3.10]{B}).

It is known (see \cite{B10}) that if $\mu$ is Skorohod or Fomin
differentiable along $h$,
then there exist conditional measures $\mu^{y,h}$ differentiable along $h$ in the same sense
and
$$
d_h\mu= d_h\mu^{y,h} |\mu|_Y(dy);
$$
 if in place of $|\mu|_Y$ we use a measure
$\sigma\ge 0$ on $Y$ with $|\mu|_Y\ll \sigma$, then we obtain
$d_h\mu= d_h\mu^{y,h,\sigma}\, \sigma(dy)$.
It follows from this that the projection of the  measure $|d_h\mu|$ on $Y$ is
absolutely continuous with respect to the projection of the
measure $\mu$, although the Skorohod derivative $d_h\mu$ itself can be singular with respect to~$\mu$
(it is  easy to construct examples on the plane).

Let us consider a Radon probability  measure
$\mu$ (not necessarily Gaussian)
on a real locally convex space $X$ with the topological dual space~$X^{*}$.
We assume that  $X$ contains a continuously and densely embedded separable Hilbert
space $H$. This embedding generates an embedding $X^{*}\to H$, since
every functional $l\in X^{*}$ defines a vector $j_H(l)\in H$ by the formula
$l(k)=(j_H(l),k)_H$, $k\in H$.
The norm in $H$ will be denoted by the symbol $|\,\cdot\,|_H$.
For example, in the case of a Gaussian measure $\mu$ one can take for $H$
the Cameron--Martin space of~$\mu$ (but other choices are possible).
A~model example is $\mathbb{R}^\infty$ with $H=l^2$ (here $X^*$ is the space of all
sequences with finitely many nonzero coordinates).

We assume below the measure $\mu$ is Fomin
differentiable along all vectors in~$H$ and that
for every fixed $h\in H$ the continuous versions of the conditional densities
on the straight lines $x+\mathbb{R}h$ are positive.
Below these densities are denoted by $\varrho^{x,h}$ without indicating~$\mu$.

We recall the definitions of the variation and semivariation of vector measures
(see \cite{DS} or \cite{DU}).
A vector  measure with values in a separable Hilbert space $H$
is an $H$-valued countably additive function $\eta$ defined on a  $\sigma$-algebra
$\mathcal{A}$ of subsets of a  space~$\Omega$.
Such a  measure automatically has finite semivariation defined by the formula
$$
   V(\eta) := \sup\Bigl|\sum_{i=1}^n \alpha_i \eta(\Omega_i)\Bigr|_H,
$$
where sup is taken over all finite partitions of $\Omega$ into disjoint
parts $\Omega_i \in \mathcal{A}$ and all finite collections of real numbers $\alpha_i$ with
$|\alpha_i|\le 1$. In other words, this is the supremum of the variations
of the scalar measures $(\eta, h)_H$
over $h\in H$ with $|h|_H\le 1$. However, this does not mean that the
variation of the vector  measure $\eta$ is finite, which is defined as
$$
   {\rm Var}(\eta): = \sup \sum_{i=1}^n |\eta(\Omega_i)|_H,
$$
where sup is taken over all finite partitions of $\Omega$ into disjoint
parts $\Omega_i \in \mathcal{A}$. The variation of a measure $\eta$ will be denoted by $\|\eta\|$
(but in \cite{DS} this notation is used for the semivariation).
It is easy to give an example of a measure with values
in an infinite-dimensional Hilbert space having bounded semivariation and
infinite variation: take $\sum_{n=1}^\infty n^{-1}\delta(e_n)e_n$,
where $\{e_n\}$ is an orthonormal basis and $\delta(e_n)$ is Dirac's measure at~$e_n$.

Let us recall a definition from \cite{BR}.

\begin{definition}
The class $SBV_H(\mu)$  consists of all functions $f\in L^1(\mu)$ for which
$$\sup_{|h|\le 1} |f\beta_h|_{L^1(\mu)}<\infty$$
and there exists an $H$-valued measure $\Lambda f$ of bounded semivariation such that
the Skorohod derivative  $d_h(f\mu)$ exists and equals $(\Lambda f,h)_H+ f \beta_h\mu$
for each $h\in H$.

The subclass  $BV_H(\mu)\subset SBV(\mu)$ consists of all functions $f$ for which $\Lambda f$ has
bounded variation.
\end{definition}

For example, if $f$ belongs to the Sobolev class $W^{1,1}(\mu)$, where $\mu$ is a Gaussian measure,
then $f\in BV_H(\mu)$; here for the measure $\Lambda f$ we can take the measure with the vector
density $D_H f$ with respect to~$\mu$. Even for a Gaussian measure $\mu$ (with an infinite-dimensional
support) the class $SBV_H(\mu)$ is strictly larger than $BV_H(\mu)$, see an example
in \cite{BR}.

It is important to note that the measure $(\Lambda f,h)_H$ can be singular with respect to~$\mu$
(say, have atoms in the one-dimensional case), but it also admits
a disintegration
$$
(\Lambda f,h)_H= (\Lambda f,h)_H^{y,h,\mu_Y}\, \mu_Y(dy)
$$
with some measures $(\Lambda f,h)_H^{y,h,\mu_Y}$ on the straight
lines $y+\mathbb{R}h$, where $y\in Y$ and $Y$~is a closed hyperplane complementing $\mathbb{R}h$.
Indeed, we have
$$
(\Lambda f,h)_H=d_h(f\mu) - f \beta_h\mu,
$$
 where the projections $|d_h(f\mu)|$ and $|f \beta_h|\mu$
on $Y$ are absolutely continuous  with respect to
the projection of $\mu$, since, as noted above,
the projection of the measure $|d_h(f\mu)|$ is absolutely continuous with respect to
the projection of $|f|\mu$.

For an interval $J$ on the real line let $BV_{loc}(J)$ denote the class of all functions on~$J$
having bounded variation on every compact interval in~$J$.

\begin{lemma}\label{lem3.1}
A function $f\in L^1(\mu)$ belongs to $SBV_H(\mu)$ precisely when there is an $H$-valued measure
$\Lambda f$ of bounded semivariation such that for every $h\in H$ for $\mu$-almost every $x$ the function
$t\mapsto f(x+th)$ belongs to $BV_{loc}(\mathbb{R})$ and its generalized derivative is
$$
(\Lambda f,h)_H^{x,h,\mu}/\varrho^{x,h}(t).
$$
A similar assertion is true for $BV_H(\mu)$,
where the measure $\Lambda$ must have bounded variation.
\end{lemma}
\begin{proof}
If the aforementioned condition is fulfilled, then the measure $(\Lambda f,h)_H+ f\beta_h\mu$
serves as the Skorohod derivative of the measure $d_h(f\mu)$. With the help of conditional measures
this is deduced from the fact that if on the real line a function
 $f$ has a locally bounded variation and a function  $g$ is locally absolutely
 continuous, then $(fg)'=g f'+ fg'$ in the sense of generalized functions.
Conversely,
if $f\in SBV_H(\mu)$, then almost all conditional measures for
$f\mu$ on the straight lines $y+\mathbb{R}h$, where $y\in Y$,
from the representation
$$
f\mu=f\mu^{y,h}\, \mu_Y(dy)
$$
 are Skorohod differentiable.
Hence almost all functions
$$
t\mapsto f(y+th)\varrho^{y,h}(t)
$$
 have bounded variation,
whence by the positivity of the conditional densities for $\mu$ we obtain that the functions
$t\mapsto f(y+th)$ have locally bounded variation.
In addition, the generalized derivative $\psi$ (i.e. $\partial_t f(y+th)$)
of such a function satisfies the
equality
$$
\partial_t (f(y+th)\varrho^{y,h}(t))=\varrho^{y,h}(t)\psi+ f(y+th)\partial_t \varrho^{y,h}(t),
$$
whence the required relationship for the derivative follows
by the equality (in the sense of generalized functions)
$$
\partial_t (f(y+th)\varrho^{y,h}(t))=(\Lambda f,h)_H^{y,h,\mu_Y}+f(y+th)\partial_t \varrho^{y,h}(t),
$$
which, in turn, follows by the equality $d_h(f\mu)=(\Lambda f,h)_H+ f \beta_h$ written in terms
of the conditional measures taken with respect to the measure $\mu_Y$ on~$Y$; here it is
 important to use a common reference measure on~$Y$. The case of $BV_H(\mu)$ is similar.
\end{proof}

Now let $U\subset X$ be a Borel set that is $H$-convex and $H$-open,
that is (see Section~1), all sets $(U-x)\cap H$ are convex and open in~$H$.
For example, this can be a set that is convex and open  in $X$.

For any $H$-convex and $H$-open set $U$ the one-dimensional sections
 $$
 U_{x,h}:=U\cap (x+\mathbb{R} h),
 $$
 are open intervals on the straight lines $x+\mathbb{R} h$. We shall often identify these intervals
 with the intervals
$$
 J_{x,h}:=\{t\in\mathbb{R}\colon\ x+th\in U\}.
 $$
 In particular, when speaking about functions on intervals $U_{x,h}$ we shall mean sometimes functions
 of the real argument on~$J_{x,h}$.

The symbol $L^p(U,\mu)$ will denote
the space of  equivalence classes of all $\mu$-measurable functions $f$ on~$U$
for which the functions $|f|^p$ are integrable with respect to the measure $\mu$ on~$U$.

Let $M_H(U,\mu)$ denote the class of all  functions $f\in L^1(U,\mu)$ such that
 $$
\|f\|_M:=\|f\|_{L^1(U,\mu)}+\sup_{|h|\le 1} \int_U |f(x)|\, |\beta_h(x)|\, \mu(dx)<\infty.
$$
The corresponding space of equivalence classes will be denoted by the same symbol.

\begin{lemma}
The set $M_H(U,\mu)$ is a Banach space with the norm
$$
\|f\|_M:=\|f\|_{L^1(U,\mu)}+ \sup_{|h|_H\le 1} \|f\beta_h^\mu\|_{L^1(U,\mu)}.
$$
\end{lemma}
\begin{proof}
Let us observe that the operator $h\mapsto f\beta_h$ from $H$ to $L^1(U,\mu)$
is linear and has a closed graph.
Indeed, suppose that $h_n\to h$ in $H$ and $f\beta_{h_n}\to g$ in $L^1(U,\mu)$.
By the continuity of the embedding $H\to D(\mu)$ we have
 $\beta_{h_n}\to \beta_h$ in $L^1(\mu)$, whence it follows that
 $f\beta_{h_n}\to f\beta_h$
in measure on~$U$, hence $g=f\beta_h$. Therefore, for every $f\in M_H(U,\mu)$ the quantity
$\|f\|_M$ is finite. Obviously, it is a norm.
Let $\{f_n\}$ be a Cauchy sequence in $M_H(U,\mu)$. Then $\{f_n\}$  converges in $L^1(U,\mu)$ to
some function~$f$. By Fatou's theorem
$f\in M_H(U,\mu)$. In addition, $f$ is a limit of $\{f_n\}$ with respect to the norm in~$M_H(\mu)$:
if $\|f_n - f_k\|_M\le \varepsilon$
for all $n,k\ge n_1$, then $\|f_n - f\|_M\le \varepsilon$ for all $n\ge n_1$.
\end{proof}

\begin{definition}\label{def3}
We shall say that $f\in M_H(U,\mu)$ belongs to the class $SBV_H(U,\mu)$ if the
function $t\mapsto f(x+th)\varrho^{x,h}(t)$ belongs to the class
$BV_{loc}(U_{x,h})$ for every fixed $h\in H$ for almost all~$x$ and there
 exists an  $H$-valued measure $\Lambda_Uf$ on $U$ of bounded semivariation such
that, for every $h\in H$, the measure $(\Lambda_Uf, h)_H$ admits the representation
$$(\Lambda_Uf, h)_H=(\Lambda_Uf, h)_H^{x,h,\mu}\, \mu(dx),
$$
 where
the measures $(\Lambda_Uf, h)_H^{x,h,\mu}$ on the straight lines $x+\mathbb{R} h$
possess the property that
    $$(\Lambda_Uf, h)_H^{x,h}+f(x+th)\partial_t \varrho^{x,h}(t)$$
is the generalized derivative of the function $t\mapsto f(x+th)\varrho^{x,h}(t)$ on $U_{x,h}$.

The class  $BV_H(U,\mu)$ consists of all  $f\in SBV_H(U,\mu)$ such that
the measure $\Lambda_Uf$ has bounded variation.
\end{definition}

As we have warned above, the sections
$U_{x,h}$ in this definition are identified with intervals $J_{x,h}$ of the real line.

Note that the defining relation for $(\Lambda_Uf, h)_H^{x,h}$ can be stated in terms of
the functions $t\mapsto f(x+th)$ (not multiplied by conditional densities): the generalized
derivatives of these functions must be (as in Lemma~\ref{lem3.1})
$$\varrho^{x,h}(t)^{-1}(\Lambda_Uf, h)_H^{x,h}.
$$

An equivalent description of functions in $SBV_H(U,\mu)$ can be given in the form
of integration by parts if in place of the class $\mathcal{F}\mathcal{C}^\infty$
we use appropriate classes of test functions for every $h\in H$.

For any fixed vector $h\in H$ we choose a closed hyperplane $Y$ complementing $\mathbb{R} h$ and
consider the class $\mathcal{D}_h$ of all bounded  functions $\varphi$ on $X$
with the following properties: $\varphi$ is measurable with respect to all Borel measures,
for each $y\in Y$ the function $t\mapsto\varphi(y+th)$ is infinitely differentiable
and has compact support in the interval $J_{y,h}=\{t\colon\, y+th\in U\}$,
and the functions $\partial_h^n \varphi$ are bounded for all~$n\ge1$. Here
$\partial_h^n \varphi(y+th)$ is the derivative of order $n$ at the point $t$ for the function
$t\mapsto \varphi(y+th)$.

Note that $\psi\varphi\in \mathcal{D}_h$ for all $\varphi\in \mathcal{D}_h$
and $\psi\in \mathcal{F}\mathcal{C}^\infty$.

\begin{lemma}\label{lem3.4}
A function $f\in M_H(U,\mu)$ belongs to $SBV_H(U,\mu)$ precisely when
there exists an $H$-valued measure $\Lambda_Uf$ on $U$ of bounded semivariation such
that, for every $h\in H$ and all $\varphi\in \mathcal{D}_h$, one has the equality
\begin{multline*}
\int_X \partial_h\varphi(x) f(x)\, \mu(dx)
\\
=-\int_X \varphi(x) \, (\Lambda_U f,h)_H (dx)-\int_X \varphi(x) f(x)\beta_h(x)\, \mu(dx).
\end{multline*}
A similar assertion with variation in place of semivariation is true for the class $BV_H(U,\mu)$.
\end{lemma}
\begin{proof}
If $f\in SBV_H(U,\mu)$, then the indicated equality follows from the definition and
the integration by parts formula for conditional measures. Let us prove the converse assertion.
Let us fix $k\in\mathbb{N}$.
It is readily verified that the set $Y_k$ of all points $y\in Y$ such that
the length of the interval $J_{y,h}$ is not less than~$8/k$
is measurable with respect to every Borel measure. In addition, it is not difficult
to show that  there exists a function $g_k\in \mathcal{D}_h$ measurable with respect to every Borel
measure and possessing the following properties:
$0\le g_k\le 1$, $g_k(y)=0$ if the length of $J_{y,h}$ is less than~$8/k$,
$g_k(y+th)=0$ if $t\not\in J_{y,h}$ or if $t\in J_{y,h}$ and the distance from $t$ to
an endpoint of $J_{y,h}$ is less than~$1/k$, $g_k(y+th)=1$ if
$t\in J_{y,h}$ and the distance from $t$ to an endpoint of  $J_{y,h}$ is not less than~$2/k$.
It follows from our hypothesis that for all $\psi\in \mathcal{F}\mathcal{C}^\infty$
we have the equality
\begin{multline*}
\int_X \partial_h\psi(x) g_k(x)f(x)\, \mu(dx)
=-\int_X \psi(x)g_k(x) \, (\Lambda_U f,h)_H (dx) \,-
\\
-\int_X \psi(x)g_k(x) f(x)\beta_h(x)\, \mu(dx)
-\int_X \psi(x)\partial_h g_k(x) f(x)\, \mu(dx).
\end{multline*}
Therefore, the measure $fg_k\mu$ is Skorohod differentiable and
$$
d_h(fg_k\mu)=g_k (\Lambda_U f,h)_H+ fg_k\beta_h\mu+ \partial_h g_k f\mu.
$$
Using the disintegration for $fg_k\mu$ and letting $k\to\infty$, we obtain the
 disintegration for $f\mu$ required by the definition.
\end{proof}

\begin{lemma}\label{lem3.5}
If $f\in SBV_H(U,\mu)$ and $\psi\in C^1_b(\mathbb{R})$, then
$$
\psi(f)\in SBV_H(U,\mu)
$$
and for any $h\in H$ one has
$$
(\Lambda_U\psi(f), h)_H=(\Lambda_U\psi(f), h)_H^{x,h,\mu}\, \mu(dx)
$$
with
$$
(\Lambda_U\psi(f ), h)_H^{x,h,\mu} = \psi'(f)(x+th)(\Lambda_U f,h),
$$
where $\psi'(f)(x+th)$ is redefined at the points of jumps of the function $t\mapsto f(x+th)$ by
the expression $\frac{\psi(f(x+th+))-\psi(f(x+th-))}{f(x+th+)-f(x+th-)}$.
Moreover,
 $$
 V(\Lambda_U\psi(f))\leq L V(\Lambda_Uf),
 $$
where $L=\sup_{u\neq v}\frac{|f(u)-f(v)|}{|u-v|}$ is the Lipschitz
constant of~$f$.

A similar assertion is true for $BV_H(U,\mu)$.
\end{lemma}
\begin{proof}
It suffices to use conditional measures and
apply the chain rule for BV-functions in the one-dimensional case,
see, e.g., \cite[p.~188]{AFP}.
\end{proof}

\begin{theorem}
The set $SBV_H(U,\mu)$ is a Banach space with the norm
$$
\|f\|_{SBV}:=\|f\|_{M}+ V(\Lambda_U f).
$$
The set $BV_H(U,\mu)$ is a  Banach space with the norm
$$
\|f\|_{BV}:=\|f\|_{M}+ {\rm Var}(\Lambda_U f).
$$
\end{theorem}
\begin{proof}
Since the space $M_H(U,\mu)$ is complete,
every Cauchy sequence $\{f_n\}$ in $SBV_H(U,\mu)$  converges
in the $M$-norm to some function $f \in L^1(U,\mu)$. The sequence of
measures $\Lambda_Uf_n$ is Cauchy in semivariation, hence converges in the norm $V$ to some measure
$\nu$ of bounded semivariation. Applying Lemma~\ref{lem3.4} it is easy to show that $f\in SBV_H(U,\mu)$
and $\nu = \Lambda_U f$ is the corresponding $H$-valued measure. Since
$$
  \|f_n - f\|_M + V(\Lambda_U f_n - \Lambda_U f) \to 0,
$$
it follows that $f$ is a limit of $\{f_n\}$ in the norm of the space $SBV_H(U,\mu)$. The proof
of completeness of the space $BV_H(U,\mu)$ is similar.
\end{proof}

The next theorem generalizes
an analogous assertion for functions in $SBV$
on the whole  space proved in \cite{BR} and confirms the conjecture
stated in \cite{BR} about the case of~$BV$.

\begin{theorem}
The classes  $SBV_H(U,\mu)$ and $BV_H(U,\mu)$ have the following property{\rm:}
if a sequence of functions $\{f_n\}$ is norm bounded in it
 and converges almost everywhere to a function $f$, then $f$ belongs to the same class,
and the norm of $f$ does not exceed the  precise upper bound of the norms of the functions~$f_n$.

Moreover, for every fixed $h\in H$, the measures $(\Lambda_U f_n,h)_H$ converge
to $(\Lambda_U f,h)_H$ in the weak topology generated by the duality with~$\mathcal{D}_h$ {\rm(}in the case $U=X$
also with respect to the duality with $\mathcal{F}\mathcal{C}^\infty${\rm)}.
\end{theorem}
\begin{proof}
These assertions are true on the real line, since our assumption about the conditional densities means
that the density of $\mu$ is positive, so $\Lambda f_n=\varrho f_n'$,
where $f_n'$ is the generalized derivative,  and these measures
converge to $\varrho f'$ in the sense of distributions.

In the general case suppose first that $\{f_n\}$ is uniformly bounded.
Let us fix $h\in H$.
Since the measures $(\Lambda_U f_n,h)_H$ are uniformly bounded and
$(\Lambda_U f_n,h)_H=(\Lambda_U f_n,h)_H^{x,h,\mu}\, \mu(dx)$, by Fatou's theorem
the function
$$
\liminf_{n\to\infty} \|(\Lambda_U f_n,h)_H^{x,h,\mu}\|
$$
is $\mu$-integrable. In particular, it is  finite almost everywhere, hence
the restriction of the function $f$ to almost every straight line $x+\mathbb{R} h$ is
in $BV_{loc}$. Moreover, we obtain finite measures
$(\Lambda_U f,h)_H^{x,h,\mu}$ on these straight lines such that the measure
$$
(\Lambda_U f,h)_H:=(\Lambda_U f,h)_H^{x,h,\mu}\, \mu(dx)
$$
 is finite for every $h\in H$.
By the Pettis theorem (see \cite[Chapter~IV, \S 10]{DS}) we obtain an $H$-valued measure
$\Lambda_U f$. It meets the requirements in Definition~\ref{def3}.
For any  $\varphi \in \mathcal{D}_h$ by the Lebesgue dominated convergence theorem we have
$$
\int_{X} \partial_h\varphi(x) f(x) \mu(dx)=\lim\limits_{n\to\infty}
\int_{X} \partial_h\varphi(x) f_n(x) \mu(dx),
$$
$$
\int_{X} \varphi(x) f(x)\beta_h(x) \mu(dx)=\lim\limits_{n\to\infty}
\int_{X} \varphi(x) f_n(x)\beta_h(x) \mu(dx).
$$
Therefore, the integrals of $\varphi$ with respect to the measures
$(\Lambda_U f_n,h)_H$ converge. Moreover, the limit is the integral of
$\varphi$ with respect to $(\Lambda_U f,h)_H$, which follows by the one-dimensional
case applied to the conditional measures. This completes the proof in the case
of $SBV$ and bounded $\{f_n\}$.

In the case of $BV$ it is necessary to show that
$\Lambda_U f$ has bounded variation. The measures $\Lambda f_n$ possess $H$-valued vector
densities $R_n$ with respect to some common nonnegative measure~$\nu$
and the sequence of functions $|R_n|_H$ is bounded in $L^1(\nu)$.
It suffices to show that for every Borel mapping $v$
 such that $|v|_H\le 1$ and $v=\sum_{i=1}^k v_i h_i$, where $h_i\in H$ are constant
and orthonormal,
we have the estimate
$$
\sum_{i=1}^k \int_X v_i(x)\, (\Lambda_U f,h_i)_H (dx)\le
\sup_n {\rm Var}(\Lambda_U f_n).
$$
It is readily seen that it is enough to do this for $v$ with functions $v_i$ such
that $v_i\in \mathcal{D}_{h_i}$.
Indeed, by using convolutions we reduce the general case to the case where the function
$v_i$ has bounded derivatives of any order along the vector~$h_i$. Next, we
approximate such functions  in $L^1(|(\Lambda_U f,h_i)_H|)$ by their products with
functions $w_{i,n}\in \mathcal{D}_{h_i}$
with the following properties: \mbox{$0\le w_{i,n}\le1$,}
$w_{i,n}(y+th_i)=0$ whenever the length $\delta_{y,h,i}$ of $U_{y,h_i}$ is less than~$4/n$
and otherwise $w_{i,n}=1$ on the inner interval of length $\delta_{y,h_i}-2/n$ with the same center
as~$U_{y,h_i}$.

For such functions we have
\begin{multline*}
\sum_{i=1}^k \int_X v_i(x)\, (\Lambda_U f,h_i)_H (dx)=
\lim\limits_{n\to\infty}
\sum_{i=1}^k \int_X v_i(x)\, (\Lambda_U f_n,h_i)_H (dx)
\\
=
\lim\limits_{n\to\infty}\int_X (v(x),R_n(x))_H\, \nu(dx)
\\
\le
\sup_n \int_X |R_n(x)|_H\, \nu(dx)=\sup_n {\rm Var}\, \Lambda_U f_n.
\end{multline*}
Thus, the case of a uniformly bounded $\{f_n\}$ is considered.

Let us now proceed to the general case where $\{f_n\}$ is not uniformly bounded.
Take a smooth increasing function $\psi$ on the real line such that
$\psi(t)=t$ if $|t|\le 1$, $\psi(t)={\rm 2\,sign}\, t$ if $|t|\ge 3$, $|\psi(t)|\le |t|$,
and $|\psi'(t)|\le 1$. Let $\psi_m(t)=\psi(t/m)$.
According to Lemma~\ref{lem3.5}, for every fixed $m$ the functions
$\psi_m(f_n)$
belong to the respective (SBV or BV) class and their norms are uniformly bounded
in $n$ and $m$.
Hence the function
$\psi_m(f)$ belongs to the same class and its norm does not exceed
the supremum of the norms of $\{f_n\}$.
We shall deal with a Borel version of~$f$,  so the functions $\psi_m(f)$ are also
Borel.
The Borel sets \mbox{$B_m=\{|f|< m\}$} are increasing to $X$ and
$|\psi_m(f)|\le |f|$. Clearly, $f\in M_H(U,\mu)$ and
$\|f\|_M=\lim\limits_{m\to\infty} \|\psi_m(f)\|_M$.
Since $\psi_{m+1}(f)$ coincides with $\psi_m(f)$ on the set~$B_m$, we obtain that
the conditional measures $(\Lambda_U \psi_m(f),h)_H^{x,h}$
on the straight lines $x+\mathbb{R}h$  have a finite
setwise  limit for almost every~$x$, and the measures $\sigma^{x,h}$
obtained in the limit give rise to bounded measures
$\sigma^{x,h}\, \mu(dx)$, which can be taken for $(\Lambda_U f,h)_H$.
In the case of $BV_H(U,\mu)$ we have additionally
that the resulting vector measure $\Lambda f$ is of bounded variation.
\end{proof}

\begin{corollary}\label{c-lip}
If $f\in SBV_H(U,\mu)$ and $\psi$ is a Lipschitzian function on the real line, then
$\psi(f)\in SBV_H(U,\mu)$ and $\|\psi(f)\|_{SBV}\le C \|f\|_{SBV}$, where $C$ is a Lipschitz
constant for~$\psi$.

The same is true in the case of $BV_H(U,\mu)$.
\end{corollary}
\begin{proof}
For smooth $\psi$ this has already been noted. The general case follows
by approximation and the above theorem.
\end{proof}

\begin{theorem}
Suppose that $I_U\in SBV_H(\mu)$. Then for every function $$f\in SBV_H(U,\mu)\cap L^\infty(U,\mu)$$
its extension by zero outside of $U$ gives a function in the class $SBV_H(\mu)$.

In the opposite direction, the restriction to $U$ of every function in $SBV_H(\mu)$,
not necessarily bounded,   gives a function in $SBV_H(U,\mu)$.

If $I_U\in BV_H(\mu)$, then the analogous assertions are true for the class~$BV_H(U,\mu)$.
\end{theorem}
\begin{proof}
Let $f\in SBV_H(U,\mu)\cap L^\infty(U,\mu)$. We may assume that $|f|\le 1$.
Let us fix $h\in H$. Then we can find a version of $f$ whose restrictions to the straight lines
 $x+\mathbb{R}h$ have locally bounded variation.
Let $a_x$ be an endpoint of the interval $U_{x,h}$ (if it exists).
Then the considered version of $f$ has a limit at~$a_x$ (left or right, respectively),
bounded by $1$ in the absolute value. Defined by zero outside of $U$, the function $f$
remains a function of locally bounded variation on all these straight lines, but
at the endpoints of $U_{x,h}$
its generalized derivative may gain Dirac measures with coefficients bounded by~$1$
in the absolute value.
However, such Dirac measures (with the coefficient $1$ at the left end and the coefficient $-1$
at the right end) are already present in the derivative of the restriction of $I_U$.
Thus, after adding these point measures to
 $(\Lambda_U f,h)_H^{x,h,\mu}$, we obtain a measure that differs from $(\Lambda_U f,h)_H$ by some
 measure with semivariation not exceeding $\|(\Lambda I_U,h)_H\|$, hence is also of bounded semivariation.
Therefore, Lemma~\ref{lem3.1} gives the inclusion of the extension to~$SBV_H(\mu)$.

The fact that $f|_U\in SBV_H(U,\mu)$ for any $f\in SBV_H(\mu)$ follows by Lemma \ref{lem3.4},
since the restriction of $\Lambda f$ to $U$ serves as $\Lambda_U f$.

In the case of $BV_H(\mu)$ we also use the fact that
any $H$-valued measure of bounded variation is given by  a Bochner integrable vector density
with respect to a suitable scalar measure.
\end{proof}

\begin{proposition}
Let $U$ be a Borel convex set. Then $I_U\in SBV_H(\mu)$.
\end{proposition}
\begin{proof}
Let us fix $h\in H$.
If $U_{x,h}$ is not empty and not the whole straight line,
the generalized derivative $\sigma^{x,h}$ of the function
$t\mapsto I_U(x+th)$ is either the difference of two Dirac's measures
at the endpoints of $U_{x,h}$ or Dirac's measure (with the sign plus or minus)
at the single endpoint (if $U_{x,h}$ is a ray).
Let us define $(\Lambda I_U,h)_H^{x,h}$ by
$$
(\Lambda I_U,h)_H^{x,h}:=\sigma^{x,h}\varrho^{x,h}.
$$
We obtain a bounded measure $(\Lambda I_U,h)_H^{x,h}\, \mu(dx)$
(indeed, $\|\sigma^{x,h}\|\le 2$, $0\le I_{U_{x,h}}\le 1$,
$\|\varrho^{x,h}\|$ and $\|\partial_t \varrho^{x,h}\|$ are $\mu$-integrable),
which
defines an $H$-valued measure $\Lambda I_U$ with the properties mentioned
in Lemma~\ref{lem3.1}, which yields the desired conclusion.
\end{proof}

The simplest model example of the objects introduced
in this section is the  standard
Gaussian product-measure $\gamma$
on the countable power of the real line $\mathbb{R}^\infty$ (discussed in the previous sections).
Its Cameron--Martin space is the classical Hilbert space $H=l^2$.
For $U$ one can take any Borel convex set of positive measure for which the
intersections $(U-x)\cap H$ are open in $H$. Note that even for a bounded convex set
$U$ the indicator function $I_U$ does not always belong to $BV_H(\gamma)$;
explicit examples are given in \cite{C}.
On the other hand, the indicator of any  open convex set in the Gaussian
case belongs to $BV_H(\gamma)$ (see \cite{C}).

Note that in place of a Hilbert space $H$ embedded into $X$ Banach spaces can be used
(then vector-valued derivatives will take values in~$X^*$).

\vskip .1in

{\bf Acknowledgements}
This work has been supported by the RFBR projects
12-01-33009, 11-01-90421-Ukr-f-a, 11-01-12104-ofi-m,
Simons-IUM fellowship of  the Simons Foundation, and the SFB 701 at the
university of Bielefeld.

\end{document}